\documentclass[12pt, a4paper]{article}

\usepackage[T1]{fontenc}
\usepackage[utf8]{inputenc}
\usepackage{amssymb}
\usepackage{amsmath}
\usepackage{amsfonts}
\usepackage{amsthm}
\usepackage{latexsym}
\usepackage[mathscr]{eucal}
\usepackage{multirow}
\usepackage{color}
\usepackage{graphicx}
\usepackage{bbm}
\usepackage{soul}
\usepackage[normalem]{ulem}

\usepackage{tikz}

\usepackage{hyperref}

\usetikzlibrary{arrows, shapes, trees, patterns, calc}
\usetikzlibrary{matrix, arrows, shapes, decorations.pathmorphing, automata, backgrounds, fit}

\newtheorem{thm}{Theorem}[section] 
\newtheorem{theorem}[thm]{Theorem}

\newtheorem{lem}[thm]{Lemma} 
\newtheorem{lemma}[thm]{Lemma} 

\newtheorem{cor}[thm]{Corollary}
\newtheorem{definition}{Definition}

 \theoremstyle{remark}
\newtheorem*{remark*}{Remark} \newtheorem{remark}{Remark}

\renewcommand{\leq}{\leqslant} \renewcommand{\le}{\leq}
\renewcommand{\geq}{\geqslant} \renewcommand{\ge}{\geq}

\numberwithin{equation}{section}

\DeclareMathOperator{\dists}{dist}
\DeclareMathOperator{\diam}{diam}

\def\({\left(} 
\def\){\right)} 

\def \EE {\mathbb{E}} 
\def \PP {\mathbb{P}} 
\def \RR {\mathbb{R}} 
\def \L {{\cal L}}
\def \Rd {{\RR^d}} 
\def \cjj {C^{1,1}}
\def \eps {\varepsilon}

\def \CII {C_4} 
\def \CIV {C_5} 
\def \CXX {C_{1}}
\def \CIIId {C_3} 
\def \CIVd {C_2} 
\def \CVd {C_6}

\newcommand{\dist}[1]{ \dists ( #1 ) }
\newcommand*\pd[1]{\mathop{}\!\mathrm{d}{#1}}
\newcommand{\R}{\mathbb{R}}
\newcommand{\N}{\mathbb{N}} 
\newcommand{\E}{\mathbb{E}}
\newcommand{\WLSC}[3]{\textrm{\rm WLSC}(#1,#2,#3)}
\newcommand{\lC}{{c}}
\newcommand{\la}{{\alpha}}
\newcommand{\lt}{{\theta}}

\newcommand{\dex}{\delta_x}

\def \Y{Y}
\def \ss{M}
\def \U{\mathcal{U}}

\title{
	Estimates of gradient of $\mathcal{L}$-harmonic functions for  nonlocal operators with order $\alpha>1$. 
}

\author{
	Tomasz Grzywny\thanks{The first author was supported by the grant 2017/27/B/ST1/0133 of National Science Centre, Poland.}, 
	Tomasz Jakubowski\thanks{The second and the third author were partially supported by the grant 2015/18/E/ST1/00239 of National Science Centre, Poland.}
		\,and 	
	Grzegorz \.{Z}urek
 }

\date{}
\begin{document}

\maketitle

\begin{abstract}
We obtain Gr\"onwall type estimates for the gradient of the harmonic functions for a L\'evy operator  with order strictly larger than 1 and  minimal assumptions of its L\'evy measure. 
\end{abstract}

\section{Introduction}

We consider a L\'{e}vy operator on $\R^d$, $d\in\N$,
\begin{equation}\label{generator}
{\cal L} f(x)= \int_{\Rd} \left(f(x+z)-f(x)
-{\bf 1}_{|z|<1}(z\cdot \nabla f(x))\right)\nu(z)\pd z\,, \quad f\in C^2_b(\R^d)\,,
\end{equation}
where $\nu(z)dz$ is a L\'{e}vy measure, i.e. $\nu(z)\geq 0$ and $\int(1\wedge|z|^2)\nu(z)\pd z<\infty$.
We assume that $\int_\R\nu(z)\pd z=\infty$, $\nu$ is symmetric and comparable with radial and nonincreasing function. That is, there exist a nonincreasing function $g:(0,\infty)\mapsto [0,\infty)$ and a constant $c>0$ such that
\begin{align}\label{eq:gcompnu}
c^{-1}g(|x|)	\leq \nu(x) \leq c g(|x|), \qquad x\in\R^d \setminus \{0\}.
\end{align}
Following \cite{MR632968}, we define
$$
	h(r)=\int_\Rd \(1\wedge \frac{|x|^2}{r^2}\right)g(|x|)\pd{x}, \qquad r>0\,.
$$
We focus on the operator with order strictly larger than 1 in the sense of the  lower weak scaling conditions for its Fourier multiplier
$$
\psi(x) = \int_{\RR^d}\left (1 -\cos ( x\cdot  z )\right) \nu(dz), \,\quad x\in\Rd.
$$
One of the primary example of a mentioned class of operators is  the fractional Laplacian $\Delta^{\alpha/2}$, for $\alpha>1$.

We note that there exists a nonincreasing function $g^*$ such that $g(r) \le g^*(r)$ and $g^*(r+1) \approx g^*(r)$, $r>r_0$ for some $r_0>2$. We introduce this function because we will need the majorant of $\nu$, which does not decrease very fast. Generally the function $g$ may not satisfy this property, e.g. $g(r) = e^{-r^2}r^{-d-\alpha}$, with $\alpha\in(0,2)$. Let 
$$\Y=\{u\in L^1_{loc}(\Rd): u\in L^1(\Rd,1\wedge g^*(|x|)dx)\}$$
and denote $\|u\|_\Y = \|u\|_{L^1(\Rd,1\wedge g^*(|x|)dx)}$. 
\begin{definition}
We say that $u \in \Y$ is $\L$-harmonic in $D$ if
$$(u,{\cal L} \varphi)=0,\quad \varphi\in C^\infty_c(D).$$ 
\end{definition}

We note that if $g(r)=0$ for $r>r_0$, for $\alpha$ sufficiently small,  $\L$-harmonic functions may not be differentiable (see e.g. \cite[Example 7.5]{MR3413864}) and do not satisfy Harnack inequality (see e.g. \cite[Example 5.5]{TGMK2018}). In the first example $\alpha<1/2$, while in the second one $\alpha=0$.
We say that a function $f$ satisfies the doubling condition if there is a constant $c$ such that
\begin{align}\label{con:doubling}
	c^{-1}f(2r) \le  f(r) \le c f(2r), \qquad r\in(0,1].
\end{align}

Our main result is the following theorem 
\begin{theorem}\label{Theorem1}
	Let $d \in \N$, $\alpha \in (1,2)$ and $D\subset \Rd$ be an open set. Let $\psi \in \WLSC{\alpha}{\theta}{c_\alpha}$, where $\theta = 0$ for $d=1$ and $\theta=1$ for $d \ge2$.  Furthermore, we assume that $g$ satisfies doubling condition \eqref{con:doubling}.
	Let $f\in Y$ be a non-negative function, $\L$-harmonic in $D$. There exists a constant $C= C(d,\psi)$ such that
	\begin{equation}\label{eq:egf_t2}
		|\nabla_x f(x)| \leq C \frac{f(x)}{\delta_x^D \land 1}, \qquad x\in D,
	\end{equation}
\end{theorem}
Here $\delta_x^D$ is the distance of $x$ to the boundary of $D$. For the definitions of $\WLSC{\alpha}{\theta}{c}$ we refer to the Section 2, Definition \ref{def:wlsc}.
\begin{remark}
The condition  $\psi \in \WLSC{\alpha}{\theta}{c_\alpha}$ in Theorem \ref{Theorem1}  means that $h(r)/r^\alpha$ is almost decreasing on $(0,1/\theta)$, i.e. 
$h(\lambda r)\ge c_\alpha \lambda^{\,\la} h(r)$ for $\lambda\le 1$ and  $0<r<1/\lt$ (see \eqref{eq:psi_h} and \eqref{eq:LSC}). 
\end{remark}
\begin{remark} If $\psi \in \WLSC{\alpha}{0}{c_\alpha}$, then \eqref{eq:egf_t2} may be strengthen to $|\nabla_x f(x)| \leq C f(x)/\delta_x^D$.  
\end{remark}

In fact, in Theorem \ref{TheoremMVP} we show \eqref{eq:egf_t2} for the larger class of functions satisfying the mean value property (abbreviated MVP, see Definition \ref{def:MVP}). In Appendix we present some relations between $\L$-harmonic and MVP functions, in particular for $f \in \Y$ both definitions are equivalent.

The assertion of Theorem \ref{TheoremMVP} is known for harmonic functions with respect to several types of operators and  under various assumptions, see e.g. \cite{MR1683048, MR1936936, Kulczycki2012,  MR1110161}.  
In the  paper \cite{MR3413864}, Kulczycki and Ryznar proved \eqref{eq:egf_t2} under certain additional assumptions on the density of the L\'evy measure. They assumed that $\nu(x)$ is positive and radial, with an absolutely continuous and decreasing profile $g$ such that $-g'(r)/r$ is nonincreasing. Furthermore $g$ had the doubling property \eqref{con:doubling} and $g(r)\approx g(r+1)$ for large $r$. These  assumptions are satisfied, e.g.,  for a large class of subordinated Brownian motions. In our paper $\nu$ is only comparable with the decreasing function which is bounded by $g^*$ possessing similar properties to the function $g$ from \cite{MR3413864}.

A very natural approach to prove estimates \eqref{eq:egf_t2} is to use the formula \eqref{eq:hvPoiss} and differentiate the Poisson kernel $P_{B(0,r)}(x,z)$ under the integral. If the exact form of the Poisson kernel is known, \eqref{eq:egf_t2} follows easily from the estimates of $\nabla_x P_{B(0,r)}(x,z)$ (see \cite{MR1936936}). However, it is rather a rare situation and usually we do not know the formula for the Poisson kernel. Another method uses a construction of the difference process developed in \cite{MR3413864}. With the help of this process the authors get the estimates of the Green function of the ball in the form \eqref{eq:thmgradG}. By using these estimates together with the Harnack inequality they finally prove \eqref{eq:egf_t2}. We note that the unimodality of the L\'evy measure is crucial in this approach.

One of the tools used  both in our approach and in \cite{MR3413864} are the estimates of the gradient of the fundamental solution $p_t(x)$ of the operator $\L$. They are needed to get the proper bounds for the potential kernel of $\L$. If $\L = -\varphi(-\Delta)$, where $\varphi$ is a Bernstein function (i.e. $\L$ is a generator of the subordinated Brownian motion) such estimates are given by the transference property expressing $\nabla p_t(x)$ in terms of the transition density of the process in dimension $d+2$ (see e.g. \cite{MR2283957}). This transference property was generalized in \cite{MR3413864} to the large class of L\'evy process and was used in the construction of the mentioned difference process. We refer also  to \cite{KK.PS.2013} and  \cite{10.3150/11-BEJ375} for other developments in this direction. In our paper, we use the estimates of $\nabla p_t(x)$ from \cite{TGKS2021} proven for pure jump L\'evy process with $\nu$ satisfying \eqref{eq:gcompnu} and $\psi \in \WLSC{\alpha}{\theta}{c}$.

We emphasize that in our approach we do not need the estimates of the Green function near the boundary of the domain as in \cite{MR3413864}. The main technical tool we use in the proof of Theorem \ref{TheoremMVP} are the $L^1$-type estimates of the gradient of the Green function of the ball obtained in Lemma \ref{lem:partg}. In order to get Lemma \ref{lem:partg}, we only need the estimates of the gradient of the fundamental solution (see the proof of Lemma \ref{lem:gradP}), which yield some preliminary estimates of $G_{B(0,r)}$ and $\nabla G_{B(0,r)}$ given in Lemmas \ref{lem:upperestgrG} and \ref{lem:lowerestG}.

Let us note that in \cite{MR3413864} Corollaries \ref{cor:Harnack} and \ref{cor:GreenEst} stated below were important tools in the proof of \eqref{eq:egf_t2}, while in our approach they are simple consequences of Theorem \ref{TheoremMVP}.
Namely, by the Gr\"onwall lemma and Theorem \ref{TheoremMVP} we get the scale invariant Harnack inequality. 
\begin{cor}\label{cor:Harnack}
Let $d \in \N$, $\alpha \in (1,2)$ and $D\subset \Rd$ be an open set. Let $\psi \in \WLSC{\alpha}{\theta}{c_\alpha}$, where $\theta = 0$ for $d=1$ and $\theta=1$ for $d \ge2$.  Furthermore, we assume that $g$ satisfies doubling condition \eqref{con:doubling}. There exists a constant $C$ such that for any $x_0\in\RR^d$, $r\in(0,1]$, and any function $f$ nonnegative on $\RR^d$ and satisfying MVP in a ball $B(x_0, r)$,
\begin{align*}
\sup_{x\in B(x_0,r/2)} f(x) \le C \inf_{x\in B(x_0,r/2)} f(x).
\end{align*}
\end{cor}
Furthermore, since the  Green function $G_D(\cdot,y)$ satisfies MVP property inside $D\setminus\{y\}$, Theorem \ref{TheoremMVP} yields the following estimates for the gradient of the Green function.
\begin{cor}\label{cor:GreenEst}
Suppose the assumptions of Corollary \ref{cor:Harnack} holds. If $\int_\Rd \frac{1}{\psi(\xi)} \pd{\xi} = \infty$, we additionally assume that $D$ is bounded. Then, there exists a constant $C$
such that
	\begin{align}\label{eq:thmgradG}
|\nabla_x G_D(x,y)| \le C \frac{G_D(x,y)}{\dex^D \land |x-y|\land 1}, \qquad x,y \in D.
\end{align}
\end{cor}
\noindent  We note that if $\int_\Rd \frac{1}{\psi(\xi)} \pd{\xi} = \infty$ and $D$ is unbounded, the Green function may not exist.

The paper is organized as follows. In Section 2, we provide the necessary definitions and prove auxiliary results on the Green function. In Section 3 we state and prove Theorem \ref{TheoremMVP}. In Appendix we show some relations between functions satisfying MVP property and $\L$-harmonic functions, which yield Theorem \ref{Theorem1}.

\section{Preliminaries}

\subsection{Notation}
In what follows, $\Rd$ denotes the Euclidean space of real numbers, $\pd{x}$ stands for the Lebesgue measure on $\Rd$. 
Without further mention we will only consider Borel sets, measures and functions in $\Rd$. 
As usual, we write $a \land b = \min(a,b)$ and $a \vee b = \max(a,b)$. For $r > 0$, 
we let $B(x,r)=\{y\in \Rd \colon |x-y|<r\}$. We denote $B_r = B(0, r)$. For the arbitrary set $A\subset \RR$, the distance to the boundary of $A$ will be denoted by
$$\delta_x^A =\dist{x,\partial A}.$$
To simplify the notation, while referring to the set $D$, we will omit the superscript, i.e.,  
$$\dex = \dex^D = \dist{x,\partial D}.$$
When we write $f (x) \approx g(x)$, we mean that there is a number $0 < C < \infty$ independent of $x$, i.e. a constant, such that for every $x$ we have $C^{-1} f (x) \le g(x) \le C f (x)$. 
The notation $C = C(a, b, \ldots, c)$ means that $C$ is a constant which depends only on $a, b, \ldots , c$.
We use a convention that constants denoted by capital letters do not change throughout the paper. 
For a radial function  $f:\Rd\rightarrow [0,\infty)$ we shall often write $f(r)=f(x)$ for any $x \in \Rd$ with $|x| = r$.

\subsection{Fundamental solution for $\L$}\label{sec:wsc}
We define 
$$
	K(r) = \frac{1}{r^2}\int_{|x|<r}|x|^2g(|x|)\pd{x},\qquad r>0\, .
$$
Clearly $K(r) \le h(r)$. Let us notice that
$$
	h(\lambda r)\leq h(r)\leq \lambda^2  h(\lambda r), \quad \lambda>1.
$$
We define the function $V$ as follows,	
$$V(0)=0 \,\,\,\mathrm{and}\,\,\, V(r)=1/\sqrt{h(r)}, \quad r>0.$$
Since $h(r)$ is non-increasing to 0, $V$ is non-decreasing and unbounded. We have
\begin{equation}\label{subadd} 
V(r)\leq V(\lambda r)\leq \lambda V(r),\quad r\geq 0 ,\,\lambda >1. \qquad 
\end{equation}
Let $\psi^*(r)=\sup_{|x|\leq r}\psi(x)$. Since $g$ in \eqref{eq:gcompnu} is nonincreasing, by \cite[Lemma 1 and (6)]{MR3165234}, for certain $\CXX$
\begin{equation}\label{eq:psi_h}
	2^{-1}\psi(\xi)\leq 2^{-1}\psi^*(|\xi|)\leq  h(1/|\xi|)\leq \CXX \psi^*(|\xi|)\leq \pi^2\CXX \psi(\xi),\quad \xi\in\Rd.
\end{equation}
\begin{definition}\label{def:wlsc}
Let $\lt\in [0,\infty)$ and
 $\phi$ be a non-negative non-zero  function on $(0,\infty)$.
We say that
$\phi$ satisfies {the} {\it \bfseries weak lower scaling condition} (at infinity) if there are numbers
$\la>0$ and  $\lC \in(0,1]$  such that
\begin{equation}\label{eq:LSC}
 \phi(\lambda r)\ge
\lC\lambda^{\,\la} \phi(r)\quad \mbox{for}\quad \lambda\ge 1, \quad r>\lt.
\end{equation}
In short, we say that $\phi$ satisfies WLSC($\la, \lt,{\lC}$) and write $\phi\in\WLSC{\la}{ \lt}{\lC}$.
If $\phi\in\WLSC{\la}{0}{\lC}$, then we say
that $\phi$ satisfies the {\bfseries  \emph{global} weak lower scaling condition}.
\end{definition}

\noindent Notice that if $\psi^*\in \WLSC{\beta}{\theta}{C}$  there is a constant  $c_1 = c_1(C, \CXX)$ 
such, that
\begin{equation}\label{scalV2}
	\frac{V(\eta r)}{V(r)} \leq c_1 \eta^{\beta / 2}
\end{equation}
for $r < 1/\theta$ and $\eta < 1$.
Moreover, if $\theta=0$ 
\begin{equation}\label{scall}
	\frac{V(\lambda r)}{V(r)}\geq c_1^{-1}\lambda^{\beta/2},\quad r> 0,\lambda >1.
\end{equation}

By similar argumentation as in \cite[Remark 4]{MR3165234} in regard to \eqref{scalV2}, we have for $r < V(1/\theta)$ and $\eta < 1$

\begin{equation}\label{skV1g}
	c_2 \eta^{2 / \la_1} \leq \frac{V^{-1}(\eta r)}{V^{-1}(r)}
\end{equation}
where $c_3 = c_3(\la_1, \lC_1, \CXX)$. 
\begin{remark} \rm
We note that the range of $r$ in \eqref{scalV2} and \eqref{skV1g} may be increased to any interval $(0,\ss)$ in the expense of the constants $c_1$ and $c_2$. The dependence of constants on $\alpha,c_\alpha$ and $\ss$ will be denoted by $\sigma$, i.e. $\sigma = \sigma(\alpha,c_\alpha,\ss)$
\end{remark}

If $\psi^* \in \WLSC{\la}{1}{\lC}$ with $\la>0$ the operator ${\cal L}$ possesses the heat kernel $p$, where 
\begin{equation}\label{eq:Forier_inv}p(t,x,y)=p_t(y-x)=(2\pi)^{-d}\int_\Rd e^{-t\psi(\xi)} \cos ((x-y) \cdot \xi)\pd{\xi}, \quad x,y\in\Rd.\end{equation}
Furthermore $p_t$ is smooth.
We use the next lemma to show existence of the fundamental solution for $\L$.

\begin{lem}\label{lem:gradP} Assume that $\psi^* \in \WLSC{\la}{1}{\lC}$ with $\la>1/2$.
	Then there exists a constant $\CIVd>0$ such that, for any  $x \in \Rd$,
	$$
		\int_0^\infty |\nabla_x p_t(x)| \pd{t} \leq \CIVd \frac{|x|V^2(|x|\wedge 1)}{(|x|\wedge 1)^{d+2}}.
	$$
\end{lem}
\begin{proof}
By \cite[Theorem 5.2]{TGKS2021} we have, for $t\leq V^2(1)$ and $x\in\Rd$, 
$$|\nabla_x p_t(x)|\leq c_1 \frac{t}{V^{-1}(\sqrt{t})}\frac{K(|x|)}{|x|^d}.$$
Hence, for $T=V^2(|x|\wedge 1)$,
\begin{align}
	\int_0^T |\nabla_x p_t(x)| \pd{t} 
	& \leq \int_0^T c_1 \frac{t}{V^{-1}(\sqrt{t})}\frac{K(|x|)}{|x|^d}\pd{t}\nonumber \\
	& 
	\leq \frac{c_2K(|x|)}{|x|^d}  \frac{T^{1/\la}}{V^{-1}(\sqrt{T})}\int_0^T  t^{1 - 1/\la} \pd{t} \nonumber\\
	& = \frac{c_3K(|x|)}{|x|^d}  \frac{T^2}{V^{-1}(\sqrt{T})} \nonumber\\
	&\leq c_3\frac{V^2(|x|\wedge 1)}{|x|^{d}(|x|\wedge 1)},\label{eq:gradP1}\end{align}
where in the last line we used $K(r)\leq h(r )$.
Let us observe that \eqref{eq:Forier_inv} implies
  $$|\nabla_x p_t(x)|\leq |x| (2\pi)^{-d}\int_\Rd e^{-t\psi(\xi)} |\xi|^2 \pd{\xi}, \quad x\in\Rd.$$
 Since $V(s)\leq V(|x|)(1+s/|x|)$, for $s\geq 0$,
  \begin{align*}
	\int_T^\infty |\nabla_x p_t(x)| \pd{t} 
	& \leq  |x| (2\pi)^{-d}\int^\infty_T\int_\Rd e^{-t\psi(\xi)} |\xi|^2 \pd{\xi}\pd{t} \\
	& |x| (2\pi)^{-d}\int_\Rd e^{-T\psi(\xi)} \frac{|\xi|^2}{\psi(\xi)} \pd{\xi}\\ 
	&\leq  c_4|x|  \int_\Rd e^{-T\psi(\xi)} |\xi|^2 V^2(1/|\xi|)\pd{\xi}\\
	& \leq  c_4|x|  V^2(|x|\wedge 1)\int_\Rd e^{-T\psi(\xi)} (|\xi|^2 +(|x|\wedge1)^{-2})\pd{\xi}\\
	&\leq c_5|x|  V^2(|x|\wedge1)\left(\frac{1}{[V^{-1}(\sqrt{T})]^{d+2}}+\frac{1}{(|x|\wedge1)^{2}[V^{-1}(\sqrt{T})]^d}\right)\\
	&=2c_5|x|  V^2(|x|\wedge1)\frac{1}{[|x|\wedge 1]^{d+2}},\end{align*}
 where the last inequality is a consequence of \cite[Proposition 3.6 and Theorem 3.1]{TGKS2020}.
\end{proof}
\begin{cor}\label{cor:potfinite}
Let $\mathbf{1} = (1,0,\ldots,0)$. For $x\in \R^d \setminus\{0\}$ we have
\begin{align*}
\int_0^\infty |p_t(x) - p_t(\mathbf{1})| dt < \infty.
\end{align*}
\end{cor}
\begin{proof}
Let $x\not=0$. By the symmetry of $p_t$ we may and do assume that the first coordinate of $x$ is nonnegative.  Due to Lemma \ref{lem:gradP} and mean value theorem we have
\begin{align}
\int^\infty_0|p_t(x)-p_t(\mathbf{1})|\pd{t}&\leq\int^\infty_0\int^1_0|(x-\mathbf{1})\cdot\nabla p_t(\mathbf{1}+s(x-\mathbf{1}))|ds\pd{t}\nonumber\\
&\leq  \CIVd |x-\mathbf{1}|\int^1_0\frac{|\mathbf{1}+s(x-\mathbf{1})|V^2(|\mathbf{1}+s(x-\mathbf{1})|\wedge 1)}{(|\mathbf{1}+s(x-\mathbf{1})|\wedge 1)^{d+2}}\pd{s}\nonumber\\
&\leq \CIVd |x-\mathbf{1}|\frac{(|x|+1)V^2(|x|\wedge 1)}{(|x|\wedge 1)^{d+2}}.\label{eq:diffP_t}
\end{align}
\end{proof}
Corollary \ref{cor:potfinite} together with the symmetry of $p_t$ let us define a fundamental solution $G(x)$ of ${\cal L}$. Namely, let $\mathbf{1} = (1,0,\ldots,0)$ and define
\begin{align*}
G(x) = 
\begin{cases} \displaystyle \int^\infty_0 p_t(x)\pd{t},&\mbox{if it is finite almost everywhere,}\\[10pt]
\displaystyle \int^\infty_0(p_t(x)-p_t(\mathbf{1}))\pd{t}, &\mbox{otherwise.}
\end{cases}
\end{align*}
The first integral above is finite if and only if $\int_{B(0,1)} \frac{1}{\Psi(\xi)} \pd\xi <\infty$ (see \cite[Theorem 37.5]{MR1739520}). In particular it happens when $d\ge3$ (see \cite[Theorem 37.8]{MR1739520}). Note that if $\int_0^\infty p_t(x)\pd{t} = \infty$ on the set $A$ of the positive Lebesgue measure, by the Chapman-Kolmogorov equation it is infinite everywhere. Indeed,
\begin{align*}
\int_0^\infty p_t(y)\pd{t} \ge \int_1^\infty \int_A  p_1(y-x)p_{t-1}(x)\pd{x}\pd{t} = \infty, \qquad y \in \RR^d.
\end{align*}

\begin{lemma}\label{lem:gradG}
Assume that $\psi^* \in \WLSC{\la}{1}{\lC}$ with $\la>1/2$. Then, $G$ is differentiable on $\Rd\setminus\{0\}$  and there exists $\CIIId$ such that, for $|x|\leq 1$,
$$|\nabla G(x)|\leq \CIIId \frac{V^2(|x|)}{|x|^{d+1}}.$$
Furthermore if $d\geq 2$ or $\psi^* \in \WLSC{\la}{0}{\beta}$ for some $\beta>0$, 
\begin{equation}\label{eq:gradG}|\nabla G(x)|\leq \CIIId \frac{V^2(|x|\wedge 1)}{(|x|\wedge1)^{d+1}},\quad x\in\Rd.\end{equation}
\end{lemma}

\begin{proof}
Let $x\neq 0$ and  $|h|<|x|/2$. Observe that Lemma \ref{lem:gradP} and \eqref{subadd} allow use us the Fubini theorem and get
$$G(x+he_i)-G(x)= \int^\infty_0\int^{h}_0 \partial_{x_i}p_t(x+ue_i) \pd{u}\pd{t}=\int^{h}_0 \int^\infty_0\partial_{x_i}p_t(x+ue_i)  \pd{t}\pd{u}.$$
Hence,
$$\nabla G(x)= \int^\infty_0\nabla_x p_t(x)  \pd{t}.$$
Another application of Lemma \ref{lem:gradP} gives us the upper bound for $ |\nabla G|$ if $|x|\leq1$.
Now we consider $|x|\geq 1$. 
Let us observe that \eqref{eq:Forier_inv} implies
  $$|\nabla_x p_t(x)|\leq (2\pi)^{-d}\int_\Rd e^{-t\psi(\xi)} |\xi| \pd{\xi}, \quad x\in\Rd.$$
Now the same line of reasoning as in the proof of Lemma \ref{lem:gradP} gives us for $d\geq 2$,
$$\int_{V^2(1)}^\infty |\nabla_x p_t(x)| \pd{t}
	\leq  c  V^2(1)\int_\Rd e^{-V^2(1)\psi(\xi)} (|\xi|^2 +|\xi|^{-1})\pd{\xi}<\infty.$$
This together with \eqref{eq:gradP1} yield the boundedness of $|\nabla G|$ on $B^c(0,1)$ in this case.

Under global lower scaling condition, by \cite[Proposition 3.6 and Theorem 3.1]{TGKS2020} and \eqref{subadd} we obtain
$$\int_{V^2(|x|)}^\infty |\nabla_x p_t(x)| \pd{t} 
	 \leq  c|x|  V^2(|x|)\int_\Rd e^{-V^2(|x|)\psi(\xi)} (|\xi|^2 +|x|^{-2})\pd{\xi}\\
	\leq c  \frac{V^2(|x|)}{|x|^{d+1}}\leq c V^2(1)$$
Let us notice that, since $V(\lambda r)\leq \lambda V(r)$, for $\lambda\geq 1$, $r>0$ we have $\lambda V^{-1}(r)\leq V^{-1}(\lambda r)$. 
Hence
\begin{align*}
	\int_{V^2(1)}^{V^2(|x|)} |\nabla_x p_t(x)| \pd{t} 
	& \leq c\frac{K(|x|)}{|x|^d}  \int_{V^2(1)}^{V^2(|x|)}  t\frac{V^{-1}(V(1))}{V^{-1}(\sqrt{t})} \pd{t} \nonumber\\
	& \leq   c\frac{K(|x|)}{|x|^d} \int_{V^2(1)}^{V^2(|x|)}  t\frac{V(1)}{\sqrt{t}} \pd{t} \nonumber\\
	&\leq cV(1)\frac{V(|x|)}{|x|^{d}}\leq cV^2(1).
\end{align*}
 
\end{proof}

\subsection{Green function}
Let $p(t,x,y) = p_t(y-x)$. We consider the time-homogeneous transition probabilities
$$
P_t(x,A) =\int_A p(t, x,y)\pd{y}, \qquad t>0, x\in \R, A\subset\Rd.
$$
By the Kolmogorov and Dynkin-Kinney theorems the transition
probability $P_t$ define in the usual way Markov probability measures 
$\{\PP^x,\,x\in \Rd\}$ on the space $\Omega$ of the
right-continuous and left-limited functions $\omega :[0,\infty)\to \R$.
We let $\EE^x$ be the corresponding expectations.
We will denote by $X=\{X_t\}_{t\geq 0}$  the canonical process on $\Omega$, $X_t(\omega) = \omega(t)$. Hence,
$$
\PP(X_t \in B) = \int_B p(t,x,y)\pd{y}. 
$$
and $X_t$ is a pure-jump symmetric L\'evy process on $\Rd$ with the L\'evy-Khinchine exponent $\psi$.
For any open set $D$, we define {\it the first exit time}\/ of the process $X_t$ from $D$,
$$\tau_D=\inf\{t>0: \, X_t\notin D\}\,.$$
Now, by the usual Hunt's formula, we define the transition density of the process {\it killed}\/ when leaving $D$
(\cite{MR0264757}, \cite{MR1329992}, \cite{MR3249349}):
\begin{equation}\label{eq:Huntp}
p_D(t,x,y)=p(t,x,y)-\EE^y[\tau_D<t;\, p(t-\tau_D, x,X_{\tau_D})],\quad t>0
,\,x,y\in \Rd \,.
\end{equation}
We briefly recall some well known properties of $p_D$ (see \cite{MR3249349}).
The function $p_D$ satisfies the Chapman-Kolmogorov equations
$$
\int_\Rd p_D(s,x,z)p_D(t,z,y)\pd{z}=p_D(s+t,x,y)\,,\quad s,t>0 ,\,
x,y\in \Rd\,.
$$
Furthermore, $p_D$ is jointly continuous when $t\neq 0$, and we have
\begin{equation}\label{eq:gg}
  0\leq p_D(t,x,y)=p_D(t,y,x)\leq p(t,x,y)\,.
\end{equation}
In particular,
\begin{equation}
  \label{eq:9.5}
  \int_\R p_D(t,x,y)\pd{y}\leq 1\,.
\end{equation}
If $D$ is a ball, by the Blumenthal 0-1 law, symmetry of $p_t$, we have
$\PP^x(\tau_D=0)=1$ for every $x\in D^c$.
In particular, $p_D(t,x,y)=0$ if $x\in D^c$ or $y\in D^c$. 

We define the Green function of $X_t$ for $D$,
\begin{align}\label{def_g}
 G_D(x,y)=\int_0^\infty p_D(t,x,y)\pd{t},  \qquad x,y \in \Rd\,
\end{align}

\begin{lemma}
Let $D$ be an open set. Additionally, we assume that $D$ is bounded if $\int_0^\infty p(t,x,y) \pd{t} = \infty$.  If $(x,y) \notin D^c \times D^c$, then
\begin{equation}\label{eq:Green_pot} 
G_D(x,y)= G(x-y)-\EE^y  G(x-X_{\tau_D}).
\end{equation}
If $D$ has the outer cone property,  \eqref{eq:Green_pot} holds for every $x,y\in\Rd$.
\end{lemma}
\begin{proof}
If $ \int^\infty_0 p_t(x)\pd{t}<\infty$ almost everywhere, \eqref{eq:Green_pot} follows directly by the Hunt formula \eqref{eq:Huntp}.
Now, suppose that $ \int^\infty_0 p_t(x)\pd{t}=\infty$ everywhere.  For $\lambda>0$, we define
$$U_\lambda(x)= \int^\infty_0 e^{-\lambda t} p_t(x)d{t}.$$
By the Hunt formula \eqref{eq:Huntp}, we have 
\begin{align*}
\int^\infty_0 e^{-\lambda t} p_D(t,x,y)\pd{t} = &U_\lambda(y-x)- \EE^y\left[e^{-\lambda \tau_D}( U_\lambda(X_{\tau_D}-x))\right ]\\
=&U_\lambda(y-x)-U_\lambda(\textbf{1}) -\EE^y\left[e^{-\lambda \tau_D}( U_\lambda(X_{\tau_D}-x)-U_\lambda(\textbf{1}))\right ]\\
&+U_{\lambda}(\textbf{1})\EE^y[1-e^{-\lambda \tau_D}]
\end{align*}
Note that $\lim_{\lambda\to0}\lambda U_{\lambda}(1) =0$. Indeed, due to \eqref{eq:diffP_t} if there would be a sequence $\lambda_k \to 0$ such that $ \lim_{k\to\infty}\lambda_k U_{\lambda_k}(\textbf{1})>0$, then we would have $\lim_{k\to\infty}\lambda_k U_{\lambda_k}(x)>0$, for $x\neq 0$. Therefore 
$$\limsup_{\lambda \to 0} \lambda \int_{B_1} U_{\lambda}(x)\pd{x}>0,$$
which gives a contradiction with \cite[the proof of Lemma A.1]{MR4194536}. Since $1-e^{-\lambda \tau_D} \le \lambda \tau_D$ and $\EE^y\tau_D <\infty$, we get 
$$\lim_{\lambda\to 0}U_{\lambda}(\textbf{1})\EE^y[1-e^{-\lambda \tau_D}] = 0.$$
Therefore, due to \eqref{eq:diffP_t} and the symmetry of $p_D$, by taking $\lambda \to 0$, \eqref{eq:Green_pot} holds for $(x,y) \not\in D^c \times D^c$. If $D$ has the outer cone property, then $\PP^y(\tau_D=0) = 1$ for $y \in D^c$ and \eqref{eq:Green_pot} holds for every $x,y\in\Rd$.
\end{proof}

\begin{definition}\label{def:C11}
We say that a non-empty open $D\subset \Rd$ is of class $C^{1,1}$ at scale $r>0$
if for every $Q\in \partial D$ there are balls
$B(x',r)\subset D$ and $B(x'',r)\subset D^c$ tangent at $Q$.
\end{definition}

We note that if the boundary of $D$ is sufficiently regular (e.g. if $D$ is $C^{1,1}$ set), by the Ikeda-Watanabe formula, the $\PP^x$ distribution of $X_{\tau_D}$ is absolutely continuous with respect to Lebesgue measure in $\R^d$, \cite{MR1825650}. Its density $P_D(x,z)$ is called the \textit{Poisson kernel} and it is given by the Ikeda-Watanabe formula (\cite[Theorem 1]{MR0142153})
\begin{align}\label{eq:Poisson}
	P_D(x,z) = \int_D G_D(x,z) \nu(z-y) dz, \qquad x\in D,\; z \in \overline{D}^c.
\end{align} 
Denote $P_D[f](x) = \E^x f(X_{\tau_D})$.
If $D$ is $C^{1,1}$ domain (or more generally if $\PP^x(X_{\tau_D} \in \partial D) =0$), we have 
\begin{align*}
P_D[f](x) = \int_{D^c} P_D(x,y) f(y) \pd{y}.
\end{align*}

\begin{definition}\label{def:MVP}
We say that $u$ satisfies \textit{mean value property} (MVP) in $D$ if $P_D[|u|](x)<\infty$ and $u(x)=P_D[u](x)$. If $u$ satisfies MVP in every open set relatively compact in $D$ then we say that $u$ satisfies MVP inside $D$. 
\end{definition}
Hence, if $f$ satisfies MVP inside the open set $D$, then for any $r<\delta_x$,
\begin{align}\label{eq:hvPoiss}
	f(x) = \E^x(f(X_{\tau_{B(x,r)}})) = \int_{B(x,r)^c} P_{B(x,r)}(x,z) f(z) \pd{z}.
\end{align}
Sometimes, in the literature functions satisfying MVP are called harmonic. 
We use this terminology to distinct the mean value property from $\L$-harmonicity. Both properties are closely related as it is shown in Appendix. Note that by \eqref{eq:Huntp} and \eqref{def_g},
\begin{align}\label{eq:G1}
\int_D G_D(x,y) \pd{y} = \E^{x} \tau_D.
\end{align}
Let us observe that the Pruitt bounds \cite[see p. 954, Theorem 1 and (3.2) ibid.]{MR632968} imply existence a constant $c=c(d)$ such that for any bounded open set $D$
\begin{align}\label{eq:ExtauDest}
	c^{-1}V^2(\delta_x)\leq \mathbb{E}^x\tau_D\leq c V^2(\diam D).
\end{align}
We note that if $D$ is a ball, better estimates for $\E^x \tau_D$ may be derived (see Lemma \ref{lem:cont_s} in Appendix).
Since $g$ is nonincreasing this implies, due to \eqref{eq:Poisson} and \eqref{eq:G1}
that there exists a constant $C$ such that for any bounded  open $D$ 
\begin{equation}
\label{eq:PoissonUp}
C^{-1} V^2(\delta_x) g(\delta_z+\diam D)\leq P_D(x,z)\leq C V^2(\diam D)g(\delta_z),\quad x\in D,\, z\in \overline{D}^c
\end{equation}

\begin{lemma}\label{lem:harmG}
Let $D$ be a $\cjj$ domain. For every $x \in D$, $G_D(x,y)$ satisfies MVP inside $D\setminus\{x\}$ and satisfies MVP in $D \setminus B(x,\varepsilon)$, where $0 < \varepsilon<\delta_x$.
\end{lemma}
\begin{proof}
Let $B$ be any $C^{1,1}$ bounded domain. Since $G_B(x,y)= 0 $ if $x\in \overline{B}^c$ we have by   \eqref{eq:Green_pot} that 
$$G(x-y)=\EE^y  G(x-X_{\tau_B}), \quad x\in B.$$
That is $G$ satisfies MVP inside $\Rd\setminus\{0\}$. This and the strong Markov property implies that $G_D(x,\cdot)$ satisfies MVP inside in $D\setminus\{x\}$. 
By the strong Markov property and the MVP property of $G$, $y \mapsto \EE^y(G(x-X_{\tau_D}))$ satisfies MVP in $D\setminus B(x,\varepsilon)$. It yields that $y \mapsto G_D(x,y)$ satisfies MVP in $D\setminus B(x,\varepsilon)$.  
\end{proof}

\begin{lemma}\label{lem:upperestgrG}
Assume \eqref{eq:gradG} holds. Let $0<a<1$. Then,  there is a constant $C$ such that, for  $r\leq 1$ and $|x|<ar$,
$$|\nabla_x G_{B_r}(x,y)| \leq C \frac{V^2(|x-y|)}{|x-y|^{d+1}}.$$
\end{lemma}
\begin{proof}
By Lemma \ref{lem:gradG} and the dominated convergence theorem
$$\nabla_x G_{B_r}(x,y)=\nabla_x G(x-y)-\EE^y \nabla_x G(x-X_{\tau_{B_r}}).$$
Furthermore, Lemma \ref{lem:gradG}, \eqref{subadd} and the monotonicity of $s\mapsto V^2(s)/s^{d+1}$ implies
\begin{align*}
|\nabla_x G_{B_r}(x,y)|&\leq \CIIId \frac{V^2(|x-y|\wedge1)}{(|x-y|\wedge1)^{d+1}}+\CIIId \EE^y \frac{V^2(|x-X_{\tau_{B_r}}|\wedge1)}{(|x-X_{\tau_{B_r}}|\wedge1)^{d+1}}\\
&\leq c\frac{V^2(|x-y|)}{|x-y|^{d+1}}+\CIIId \frac{V^2((1-a)r)}{((1-a)r)^{d+1}}\\
&\leq c\frac{V^2(|x-y|)}{|x-y|^{d+1}}.
\end{align*}
\end{proof}

\begin{lemma}\label{lem:lowerestG}
Assume that $\psi^* \in \WLSC{\la}{1}{\lC}$ with $\la>0$.  Then there exists $C>0$ and $0<\kappa<1$ such that, for any $r<1$,
$$G_{B_r}(x,y)\geq C \frac{V^2(|x-y|)}{|x-y|^d}, \quad |x|,|y|\leq \kappa r.$$
\end{lemma}

\begin{remark}
If the measure $\nu$ is unimodal, i.e. its density is radial and nonincreasing, due to \cite[Theorem 1.3]{TGMK2018} and \cite[Lemma 2.3]{TGKS2020} the above inequality is equivalent to the lower scaling condition.
\end{remark}
\begin{proof}[Proof of Lemma \ref{lem:lowerestG}]

\cite[Theorem 5.2]{TGKS2021} gives us, for $t<1$ and $x\in\Rd$,
$$p_t(x)\leq c_1 t\frac{K(|x|)}{|x|^d}.$$
Furthermore,  by \cite[Corollary 5.5]{TGKS2020}, for $t\leq 1$,
$$p_t(x)\geq \frac{c_2}{[V^{-1}(\sqrt{t})]^d}, \quad V^2(|x|)\leq t.$$
Let $\lambda=\frac{c_2}{2c_1}<1$.  By \eqref{eq:Huntp}, monotonicity of $s\mapsto\frac{K(s)}{s^d}$ and the inequalities above, for $V^2(|x-y|)\leq t\leq \lambda V^2((\delta_x\vee\delta_y)$ we get
\begin{align*}
p_{B_r}(t,x,y)&\geq \frac{c_2}{[V^{-1}(\sqrt{t})]^d}- c_1 t\frac{K(\delta_x\vee\delta_y)}{(\delta_x\vee\delta_y)^d}\\
&\geq  \frac{c_2}{[V^{-1}(\sqrt{t})]^d}- c_1 t\frac{K(V^{-1}(\sqrt{t/\lambda}))}{(V^{-1}\sqrt{t/\lambda}))^d}\\
&\geq  \frac{c_2}{[V^{-1}(\sqrt{t})]^d}- c_1 t\frac{1}{V^2(V^{-1}(\sqrt{t/\lambda}))(V^{-1}\sqrt{t/\lambda}))^d}\\
&\geq\frac{c_2}{2[V^{-1}(\sqrt{t})]^d}. 
\end{align*}
Let $a<b<1$. Then, for $|x-y|\leq a r$ and $|x|,|y|\leq (1-b)r$, we have
$$G_{B_r}(x,y)\geq \int^{\lambda V^2(br)}_{V^2(ar)}\frac{c_2}{2[V^{-1}(\sqrt{t})]^d}\pd{t}\geq \frac{c_2(\lambda V^2(br)-V^2(ar))}{2(br)^d}.$$
Now it is enough to take $\tfrac{a}{b}$ small enough, such that 
$\frac{V^2(br)}{V^2(ar)}\geq (2\lambda)^{-1}$
to get 
\begin{align}\label{eq:lowerestGw}
G_{B_r}(x,y)\geq \frac{c}{b^d}\frac{V^2(br)}{r^d}\geq  \frac{c}{b^{d+2}}\frac{V^2(r)}{r^d}. 
\end{align}
Now, for given $x,y \in B_{\kappa r}$ we put $r_0 = |x-y|/2$. Eventually, by \eqref{subadd}, we get
\begin{align*}
G_{B_r}(x,y) \ge G_{B_{r_0}}(x,y) \ge  \frac{c}{b^{d+2}} \frac{V^2(r_0)}{r_0^d} \ge \frac{c2^{d-1}}{b^{d+2}} \frac{V^2(|x-y|)}{|x-y|^d}.
\end{align*}

\end{proof}

\begin{lem}\label{GradGreen1} Assume \eqref{eq:gradG}.
	Let $r\leq1$. Then, for  $f\in L_\infty(B_r)$, 
	\begin{equation}
    	\nabla_x\int_{B_r}  G_{B_r}(x,z)f(z)\pd{z} = \int_{B_r} \nabla_x G_{B_r}(x,z)f(z)\pd{z},\quad |x|<r.
	\end{equation}
\end{lem}
\begin{proof}
First, note that by \eqref{scalV2} for any $0<\rho\le \ss$, we have
\begin{align}\label{eq:V2est}
	\int_{B_{\rho}}\frac{V^2(|y|)}{|y|^{d+1}}\pd{y} = c_1\int_0^{\rho} \frac{V^2(s)}{s^2} \pd{s}
	\leq c_2\frac{V^2(\rho)}{\rho}.
\end{align}
for some constant $c_2$ depending on $\ss$. 
Denote by $e_1,\ldots,e_d$ the standard basis in $\R^d$. We fix $x \in B_r$ and 
let $0<h<\dex/2$ and $h_i=he_i\in\Rd$.
By Lemma \ref{lem:upperestgrG} with $a=(r+|x|)/2$
\begin{align*}
	\frac{| G_{B_r}(x+h_i,z)- G_{B_r}(x,z)|}{h} 
	&= \frac{1}{h}\left|\int_0^1\frac{\pd{}}{\pd{s}} G_{B_r}(x+s h_i,z) \pd{s} \right| \\
	=  \left|\int_0^1\frac{\partial}{\partial x_i} G_{B_r}(x+s h_i,z) \pd{s} \right|  
	& \leq  c\int_0^1\frac{V^2(|x+s h_i-z|)}{|x+s h_i-z|^{d+1}} \pd{s}.
\end{align*}
Since $f$ is bounded and   $r\mapsto \frac{V^2(r)}{r^{d+1}}$ is nonincreasing, the rearrangement inequality implies uniformly in $h$ integrability of $z\mapsto \frac{G_D(x+h_i,z)- G_D(x,z)}{h}f(z) $ on $B_r$, thus by \eqref{eq:V2est} we get the assertion.
\end{proof}

As in Lemma \ref{lem:harmG}, we prove the MVP property for $\nabla_x G_D(x,\cdot)$.
\begin{lemma}\label{lem:harmgrG}
Let $D$ be a $\cjj$ domain. For every $x \in D$, $\nabla G_D(x,y)$ satisfies MVP inside $D\setminus\{x\}$ and satisfies MVP in $D \setminus B(x,\varepsilon)$, where $0 < \varepsilon<\delta_x/2$. 
\end{lemma}
\begin{lem}\label{lem:nablaGP1}
	Let $0<a<b<1$, $x \in B_{ar}$, $y \in B_r \setminus B_{br}$.	Then,
	$$
		\nabla_xG_{B_r}(x,y) = \int_{B_r \setminus B_{br}} \nabla_xG_{B_r}(x,z)P_{B_r \setminus B_{br}}(y,z)\pd{z}.
	$$
\end{lem}
\begin{proof}
Let $A= B_r\setminus B_{br}$.
By Lemma \ref{lem:harmG} and \eqref{eq:Poisson}, we have
$$G_{B_r}(x,y)= \int_{B_{br}} G_{B_r}(x,z)P_{A}(y,z)\pd{z},\quad x\in B_r,\; y\in A.$$
Let us fix  $y\in A$. 
For $z \in B_{br}$ we define $P_1(y, z) = P_{A}(y, z)  \mathbbm{1}_{B_{r(a+b)/2}}(z)$ and $P_2(y,z) = P_A(y,z) - P_1(y,z)$. 
By \eqref{eq:PoissonUp}, $P_1(y,\cdot)$ is bounded. Hence, by Lemma \ref{lem:nablaGP1} we have
$$
	\nabla_x\int_{B_{br}} G_{B_r}(x,z)P_1(y,z)\pd{z} 
	= \int_{B_{br}}\nabla_x G_{B_r}(x,z) P_1(y,z)\pd{z}. 
$$
Denote by $e_1,\ldots,e_d$ the standard basis in $\R^d$. For $i \ = 1, \dots, d$ and $h \in \R\setminus \{0\}$ we denote $h_i = he_i$. 
By Lemma \ref{lem:upperestgrG} we see $\partial_{e_i} G_D(x,z)$ is bounded on the support of the function $P_2(y, \cdot)$. 
By this, mean value theorem, and dominated convergence theorem we have
\begin{align*}
	&\lim\limits_{h \rightarrow 0} \frac{1}{h}\left(\int_{B_{br}} G_{B_r}(x + h_i,z)P_2(y,z)\pd{z} - \int_{B_{br}} G_{B_r}(x,z)P_2(y,z)\pd{z}\right)\\
	= &\lim\limits_{h \rightarrow 0} \int_{B_{br}}\frac{ G_{B_r}(x + h_i,z)- G_{B_r}(x,z)}{h}P_2(y,z)\pd{z} 
	= \int_{B_{br}} \partial_{x_i} G_{B_r}(x,z)P_2(y,z)\pd{z},
\end{align*}
which completes the proof. 

\end{proof}

\section{Proof of Theorem \ref{Theorem1}}
In this chapter we focus on the gradient of functions satisfying MVP property. 
The main aim is to prove that for such functions inequality \eqref{eq:egf_t2} holds.

Let $D$ be an open set in $\Rd$, $d \in \N$ and $f$ be a nonnegative function in $\Rd$ satisfying MVP inside $D$. 
We fix $x \in D$ and $r = (\dex \land 1)/2$. 
In order to prove \eqref{eq:egf_t2}, without a loss of generality we may and do assume $x=0$.   
By Lemma \ref{lem:upperestgrG}, for every $a \in (0,1)$ there is a constant $c$ such that
\begin{equation}\label{eq:gradharm}
	|\nabla_x G_{B_r}(0,y)|
	\leq c\frac{V^2(|y|)}{|y|^{d+1}}, \qquad y\in\RR^d.
\end{equation}
Similarly, by Lemmas \ref{lem:upperestgrG} and \ref{lem:lowerestG}, there is $\kappa\in (0,1)$ such that
\begin{equation}\label{eq:grad2harm}
	|\nabla_x G_{B_r}(0,y)|
	\leq c\frac{G_{B_r}(0, y)}{|y|}, \qquad |y| \le \kappa r.
\end{equation}
Let us fix $\kappa \in(0,1)$ such that \eqref{eq:grad2harm} holds. 

\begin{lemma}\label{lem:partg}
There exist $\CII = \CII(d, \sigma)$ such that
$$
\int\limits_{B_{\rho}}|\nabla_xG_{B_\rho}(0,y)|\pd{y}
	\leq \CII\frac{V^2(\rho)}{\rho},  \qquad 0<\rho<\ss.
$$
\end{lemma}

\begin{proof}
By \eqref{eq:gradharm} and \eqref{eq:V2est}, for $0<\rho < \ss$ we get
\begin{align*}
	\int\limits_{B_{\rho}}|\nabla_xG_{B_\rho}(0,y)|\pd{y} & 
	\leq c_0\int_{B_{\rho}}\frac{V^2(|y|)}{|y|^{d+1}}\pd{y} \leq c_1\frac{V^2(\rho)}{\rho}.
\end{align*}
\end{proof}

Since $f$ satisfies MVP inside $D$, by the Ikeda-Watanabe formula \eqref{eq:Poisson} we have
\begin{align}\label{eq:harmf}
f(0) = \int_{B_r^c} P_{B_r}(0,w) f(w) \pd{w} = \int_{B_r} G_{B_r}(0, y) \int_{B_r^c} f(w)\nu(w-y)\pd{w}\pd{y}.
\end{align}
\begin{lemma}\label{lem:poCr}
Let $C_r = B_{2r}\setminus B_r$ and suppose $g$ satisfies the doubling condition \eqref{con:doubling}. There exists a constant $\CIV = \CIV(d, \sigma,\kappa)$ such that
$$
\int\limits_{B_{\kappa r}}|\nabla_x G_{B_r}(0, y)| \int\limits_{C_r} f(w)\nu(w-y)\pd{w}\pd{y} \leq \CIV \frac{f(0)}{r}.
$$
\end{lemma}
\begin{proof}
By Lemma \ref{lem:partg}, and \eqref{eq:lowerestGw}, we have
\begin{equation}\label{eq:partg}
	\int\limits_{B_{\kappa r}}|\nabla_xG_{B_r}(0,y)|\pd{y} \leq c \frac{V^2(r)}{r}
	\leq \frac{c_1}{r}\int_{B_{\kappa r}}G_{B_r}(0, y)\pd{y}.
\end{equation}
Note that there exists a constant $c_2 = c_2(d, \sigma,\kappa)$ such that
	\begin{align}\label{eq:levycomp}
	c_2^{-1} \nu(w) \le 	\nu(w-y) \le c_2\nu(w), \qquad 	r < |w| < 2r,\;  |y|<\kappa r.
	\end{align}
Indeed, since $\frac{2}{1-\kappa} |w-y| > 2r > |w| > r > |w-y|/3$, by \eqref{eq:gcompnu} and \eqref{con:doubling} we get \eqref{eq:levycomp}.
Now, by \eqref{eq:levycomp}, \eqref{eq:partg} and \eqref{eq:harmf}
\begin{align*}
  &\int_{B_{\kappa r}}|\nabla_xG_{B_r}(0, y)| \int_{C_r} f(w)\nu(w-y)\pd{w}\pd{y} 
	\leq  c_2\int_{B_{\kappa r}}|\nabla_xG_{B_r}(0, y)| \int_{C_r} f(w)\nu(w)\pd{w}\pd{y} \\
	&\leq \frac{c_3}{r}\int_{B_{\kappa r}}G_{B_r}(0, y)\pd{y} \int\limits_{C_r} f(w)\nu(w)\pd{w} 
	\leq \frac{c_4}{r}\int_{B_{\kappa r}}G_{B_r}(0, y)\int_{C_r} f(w)\nu(w-y)\pd{w}\pd{y}\\
	&\leq  c_5\frac{f(0)}{r}.
\end{align*}
\end{proof}

\begin{lemma}\label{lem:poDr}
Let $D_r = B_{2r}^c$. There exists a constant $\CVd = \CVd(d, \sigma, \kappa)$ such that
$$
\int\limits_{B_{\kappa r}}|\nabla_x G_{B_r}(0, y)| \int\limits_{D_r} f(w)\nu(w-y)\pd{w}\pd{y} \leq \CVd\frac{f(0)}{r}.
$$
\end{lemma}
\begin{proof}
Let $\rho = 7r/4$ and $a = 6\kappa/7$. Note that $\rho<1$. For a given $w \in D_r$ we let $w' \in \partial D_r$ be such that $|w-w'| = \dist{w, \partial D_r}$. 
We consider a one-sided circular cone $\gamma^w$  with apex in $0$ and axis going through the point $w$. 
The cone $\gamma^w$ is chosen in such a way that its boundary contains $\partial B(0, ar') \cap \partial B(w', (2-\kappa) r)$. 
Denote $\gamma_1^w := \left(B_{\kappa\rho} \setminus B_{a\rho}\right) \cap \gamma^w$. 
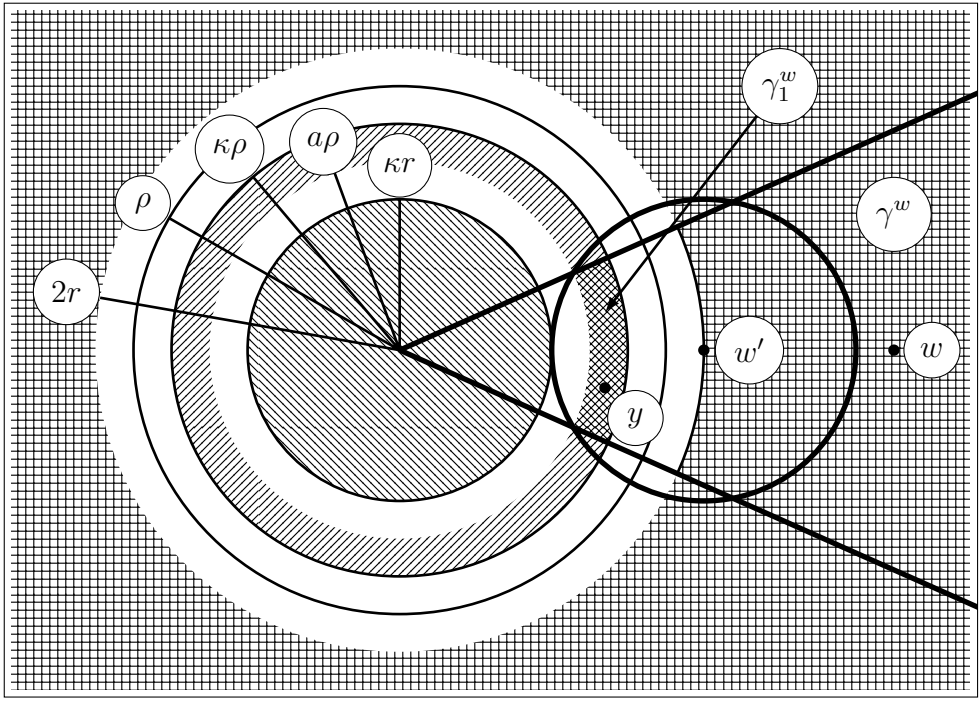
\begin{figure}
\begin{center}
\begin{tikzpicture}[scale=0.1]
\clip [draw] (-52,46) rectangle (76,-46);
\coordinate (A) at (0,0);
\coordinate (W) at (40, 0);
\coordinate (V) at (65, 0);
\coordinate (Y) at (27, -5);

\fill[pattern=grid] (-51,45) rectangle (75,-45);
\fill[white] (0,0) circle (40cm);
\draw[line width = 1pt, pattern = north east lines] (0,0) circle (30cm);

\begin{scope}
	\path [clip] (0,0) -- (22.8125*2, 10.227*2) -- (22.8125*2, -10.227*2);
	\draw[line width = 1pt, fill=white] (0, 0) circle (40cm);
	\draw[line width = 1pt, pattern = crosshatch] (0,0) circle (30cm);
\end{scope}

\draw[line width = 1pt] (0,0) circle (35cm);
\fill[black] (W) circle (.75cm);
\fill[black] (V) circle (.75cm);
\fill[black] (Y) circle (.75cm);
\fill[white] (0,0) circle (25cm);
\draw[line width = 1pt, pattern = north west lines] (0,0) circle (20cm);
\draw[line width = 2pt] (W) circle (20cm);

\draw [line width=2pt] 
	(0,0) -> (22.8125*4, 10.227*4)
	(0,0) -> (22.8125*4, -10.227*4)
	;

\draw [line width=1pt]
	(A) -- ({20*sin(0)}, {20*cos(0)})
	(A) -- ({25*sin(-20)}, {25*cos(-20)})
	(A) -- ({30*sin(-40)}, {30*cos(-40)})
	(A) -- ({35*sin(-60)}, {35*cos(-60)})
	(A) -- ({40*sin(-80)}, {40*cos(-80)})
	;
	
\draw 
	({(4.5+20)*sin(0)}, {(4.5+20)*cos(0)}) node [circle, fill=white, draw=black] {$\kappa r$}
	({(4.4+25)*sin(-20)}, {(4.4+25)*cos(-20)}) node [circle, fill=white, draw=black] {$a\rho$}
	({(4.8+30)*sin(-40)}, {(4.8+30)*cos(-40)}) node [circle, fill=white, draw=black] {$\kappa\rho$}
	({(4+35)*sin(-60)}, {(4+35)*cos(-60)}) node [circle, fill=white, draw=black] {$\rho$}
	({(4.5+40)*sin(-80)}, {(4.5+40)*cos(-800)}) node [circle, fill=white, draw=black] {$2r$}
	;

\draw
	(46, 0) node [circle, fill=white, draw=black]{$w'$}
	(70, 0) node [circle, fill=white, draw=black]{$w$}
	(65, 18) node [circle, fill=white, draw=black]{$\gamma^w$}
	(31, -9) node [circle, fill=white, draw=black] (0.8cm) {$y$}
	;
\draw [-latex, line width=1pt]	(50,35) -- (27, 5)	;
\draw 	(50, 35) node [circle, fill=white, draw=black]{$\gamma_1^w$};

\end{tikzpicture}
\end{center}
\caption{
The illustration shows two-dimensional schema of intersection of the cone $\gamma^w$ with the balls described in the proof of the Lemma \ref{lem:poDr}. 
}\label{fig:duzokulek}
\end{figure}
Now, by Lemma \ref{lem:partg}, \eqref{subadd} and \eqref{eq:lowerestGw}, we have
\begin{equation}\label{eq:part74g}
	\int_{B_{\kappa r}}|\nabla_xG_{B_r}(0,y)|\pd{y} 
	\leq	c_1 \frac{V^2(r)}{r}  \leq \frac{c_2}{r} \int_{\gamma_1^w} \frac{V^2(r)}{r^d} \pd{y}	
	\leq \frac{c_3}{r}\int_{\gamma_1^w}G_{B_{\rho}}(0, y)\pd{y},
\end{equation}
where the constant $c_3$ does not depend on $w$.
We observe also that (see Figure \ref{fig:duzokulek})
\begin{align}
	|w| - \kappa r &> |w-y|, \qquad y \in \gamma_1^w,  \label{indeg:wawy}\\
	|w| - \kappa r &< |w-y|, \qquad y \in \kappa r.  \label{indeg:wawy2}
\end{align} 
Therefore, by \eqref{eq:gcompnu}, monotonicity of $g$, \eqref{indeg:wawy}, \eqref{eq:part74g} \eqref{indeg:wawy2} and \eqref{eq:harmf} , we get
\begin{align*}
	& \int_{B_{\kappa r}}|\nabla_x G_{B_r}(0,y)| \int_{D_r} f(w)\nu(w-y)\pd{w}\pd{y} \\
	&\leq c_4\int_{B_{\kappa r}}|\nabla_x G_{B_r}(0,y)| \int_{D_r} f(w)\nu(|w|-\kappa r)\pd{w}\pd{y}\\
	&\leq \frac{c_3c_{4}}{r} \int_{D_r} \int_{\gamma_1^w}G_{B_{\rho}}(0, y) f(w)\nu(|w|-\kappa r)\pd{y}\pd{w} \\
	&\leq  \frac{c_5}{r} \int_{D_r} \int_{\gamma_1^w} G_{B_{\rho}}(0, y) f(w)\nu(w-y)\pd{y}\pd{w} \\
	& \leq \frac{c_6}{r}\int_{B_\rho}G_{B_{\rho}}(0, y) \int_{D_r}f(w)\nu(w-y)\pd{w}\pd{y}
	\leq c_7	\frac{f(0)}{r}. 
\end{align*}

\end{proof}

Now, we are ready to prove the main result. 

\begin{theorem}\label{TheoremMVP}
	Let $d \in \N$, $\alpha \in (1,2)$ and $D\subset \Rd$ be an open set. Let $\psi \in \WLSC{\alpha}{\theta}{c_\alpha}$, where $\theta = 0$ for $d=1$ and $\theta=1$ for $d \ge2$.  Furthermore, we assume that $g$ satisfies doubling condition \eqref{con:doubling}.
	Let $f$ be a function satisfying MVP inside $D$. There exists a constant $C$ such that
	\begin{equation*}
		|\nabla_x f(x)| \leq C \frac{f(x)}{\delta_x^D \land 1}, \qquad x \in D,
	\end{equation*}
\end{theorem}

\begin{proof}
Assume that $\la>1$. By Lemmas \ref{lem:poCr} and \ref{lem:poDr}  for $r\le 1/2$, we have  
\begin{align}\label{eq:poBrc}
	\int_{B_{\kappa r}}| \nabla_x G_{B_r}(0,y)|\int_{B_r^c}f(w)\nu(w-y)dwdy\leq C\frac{f(0)}{r}.
\end{align}
As before we may and do assume $x=0$. Let $r= (\dex\land1)/2$. Denote $A_r=B_{r}\setminus B_{\kappa r}$. By Lemmas \ref{lem:harmgrG} and \ref{lem:nablaGP1}, we have
\begin{align*}&\int_{A_r}| \nabla_x G_{B_r}(0,y)|\int_{B_r^c}f(w)\nu(w-y)\pd{w}\pd{y} \\
&\leq \int_{A_r}\int_{B_{\kappa r}} | \nabla_x G_{B_r}(0,u)|P_{A_r}(y,u)\pd{u}\int_{B_r^c}f(w)\nu(w-y)\pd{w}\pd{y}\\
&= \int_{C_r}| \nabla_x G_{B_r}(0,u)|W(u) \pd{u}  +  \int_{D_r}| \nabla_x G_{B_r}(0,u)|W(u) \pd{u} = I_1+I_2,
\end{align*}
where $C_r = B_{\kappa^2 r}$, $D_r = B_{\kappa r} \setminus B_{\kappa^2 r}$ and 
\begin{align*}
W(u)= \int_{A_r} P_{A_r}(y,u)\int_{B_r^c}f(w)\nu(w-y)\pd{w}\pd{y}.
\end{align*}
We need to prove that $I_1$ and $I_2$ are bounded by $cf(0)/r$. By Lemma \ref{lem:harmG}, $G_{B_r}(0,y) = \int_{\kappa r} G_{B_r}(0,u) P_{A_r}(y,u) \pd{u}$ for $y\in A_r$. Hence, by \eqref{eq:grad2harm} we get
\begin{align*}I_2&\leq \frac{C}{r} \int_{D_r} \int_{A_r} G_{B_r}(0,u)P_{A_r}(y,u)\int_{B_r^c}f(w)\nu(w-y)\pd{w}\pd{y}\pd{u}\\
&\leq \frac{C}{r}\int_{A_r} G_{B_r}(0,y)\int_{B_r^c}f(w)\nu(w-y)\pd{w}\pd{y} \leq C \frac{f(0)}{r}
\end{align*}
Next, since $f$ satisfies MVP, by \eqref{eq:Poisson} we have
\begin{align*}W(u)&=\int_{A_r} \int_{A_r}G_{A_r}(y,z)\nu(u-z)\pd{z} \int_{B_r^c}f(w)\nu(w-y)\pd{w}\pd{y}\\
&\leq\int_{A_r} \nu(u-z) \int_{A_r^c}f(w)\int_{A_r}G_{A_r}(y,z)\nu(w-y)\pd{y}\pd{w}\pd{z}\\
&=  \int_{A_r} \nu(u-z) \int_{A_r^c}f(w)P_{A_r}(z,w)\pd{w}\pd{z}\\
&= \int_{A_r} \nu(u-z)f(z)\pd{z}.
\end{align*}
Hence, by \eqref{eq:poBrc} 
$$I_1\leq \int_{ C_r}  | \nabla_x G_{B_r}(\theta,u)|\int_{B^c_{\kappa r}} \nu(u-z)f(z)\pd{z}\pd{u}\leq C\frac{f(0)}{r}.$$
\end{proof}

\begin{proof}[Proof of Theorem \ref{Theorem1}]
If $f \in L^1_{loc}(\Rd) \cap L^1(\Rd,1\wedge g^*(|x|)dx)$ is $\L$-harmonic, then by Lemma \ref{lem:harm=MVP}, $f$ satisfies MVP. Now \eqref{eq:egf_t2} follows from Theorem \ref{TheoremMVP}. 
\end{proof}

\appendix
\section{Appendix} 
In this section we will prove relation between $\L$-harmonic functions and functions satisfying MVP. 
Analogous results were proved for the fractional Laplacian in \cite{MR1671973} and operators $\L$ with twice differentiable L\'evy density (\cite{MR4194536}).
We will work with slightly different assumptions. We will not assume WLSC property for $\psi$, but we will still keep an assumption that $\int_{\R^d} g(z) \pd{z} = \infty$ (generally it follows from WLSC condition). However, to simplify the discussion, we will assume that there is  $R\in(0,\infty]$ such that  a function $y\mapsto G_{B_R}(0,y)$ is continuous on $B_R\setminus\{0\}$. 
Let us notice that by the Hunt formula \eqref{eq:Huntp} and the strong Markov property we have
\begin{equation}\label{eq:Hunt_ap}G_{B_r}(0,y)=G_{B_R}(0,y)-P_{B_r}[G_{B_R}(0,\cdot)](y),\quad y\neq0, \; 0<r<R. \end{equation}
In particular, by the proof of Lemma \ref{lem:harmG},  $G_{B_r}(0,\cdot)$ satisfies MVP in $B_r\setminus \overline{B}_{\rho}$ for every $\rho>0$.

By ${\U}$ we denote the Dynkin operator that is (see \cite[Chapter 7]{MR2962168})
$${\U}[f](x)=\lim_{r\to 0}\dfrac{P_{B(x,r)}[f](x)-f(x)}{\E^x \tau_{B(x,r)}}.$$ 
By the Markov property if $u$ satisfies MVP in an open set $D$ then ${\U}[u](x)=0$ on $D$.
Let $g^*$ be such that $g(s)\leq g^*(s)$, $s>0$,  $g^*$ satisfies doubling and $g^*(r)\approx g^*(r+1)$ for large $r$. Notice that $g^*(r)$ can vanish for large $r$.  Recall that 
$$\Y=\{u\in L^1_{loc}(\Rd): u\in L^1(\Rd,1\wedge g^*(|x|)dx)\}$$
and $\|u\|_Y = \|u\|_{L^1(\Rd,1\wedge g^*(|x|)dx)}$. 
 Observe that if $g^*$ is positive, for $\varphi\in C^\infty_c(D)$ we have
$$|{\cal L}[\varphi](x)|\leq C(1\wedge g^*(|x|)).$$
If $g^*$ vanishes for large arguments, then $|{\cal L}[\varphi](x)|\leq C$, $x\in\R^d$ and $|{\cal L}[\varphi](x)|=0$ for $|x|$ sufficiently large.

\begin{lem}\label{lem:MPVtoHarm} 
Assume that $u\in \Y$.  If $u$ satisfies MVP in an open set $D$ then $u$ is $\L$-harmonic in $D$.
\end{lem}

\begin{proof}
Let $\phi_\epsilon$, $\epsilon>0$ be a standard mollifier.  Since $\phi_\epsilon*u\in C^2$ we have
$$(\phi_\epsilon*u,{\cal L} \varphi)=({\cal L}(\phi_\epsilon*u), \varphi).$$
Since ${\U}$ is an extention of ${\L}$ (see \cite[Theorem 7.26]{MR2962168}) and it is a translation invariant, we conclude 
$${\L}[\phi_\epsilon*u](x)={\U}[\phi_\epsilon*u](x)=\phi_\epsilon*{\U}[u](x)=0$$ 
on $D_\epsilon=\{x\in D:\dist{x,\partial D}>\epsilon\}$. That is 
$\phi_\epsilon*u$ is $\L$-harmonic in $D_\epsilon$. 
Since $u\in \Y$ the distribution ${\L}u$ is properly define. Passing with $\epsilon \to 0$ we obtain (we use ${\cal L}(\phi_\epsilon*u) = \phi_\epsilon*{\cal L}u$ in distributional sense),
$$(u,{\cal L} \varphi)=0,\quad \varphi\in C^\infty_c(D_{\delta}),$$
for every $\delta>0$ which ends the proof.
\end{proof}

We will need the following maximum principle for $\U$.

\begin{lemma}\label{lem:maxpr}
Let $D_1$ be a bounded domain and $h$ a continuous function on $D_1$. Suppose $x_0 \in D_1$ is such that $h(x_0)=\sup_{x \in D_1}h(x)>0$. Then either $h$ is constant on $\mathrm{int}(\mathrm{supp}(\nu)+x_0)$ or $\U[h](x_0)<0$.
\end{lemma}
\begin{proof}
Let  $x\in  \mathrm{int}(\mathrm{supp}(\nu)+x_0)$. If $h(x)<h(x_0)$ there exists $r>0$ such that $h(x_0)-h(z)\geq (h(x_0)-h(x))/2$ for $z\in B(x,2r)\subset  \mathrm{int}(\mathrm{supp}(\nu)+x_0)$. For $s<r$, by the Ikeda-Watanabe formula and \eqref{eq:G1} we have
\begin{align*}
h(x_0)-P_{B(x_0,s)}[h](x_0)&= \int_{B(x,s)^c}P_{B(x_0,s)}(x_0,z) (h(x_0)-h(z)) \pd{z}\\&\geq (h(x_0)-h(x))/2 \int_{B(x,r)}P_{B(x_0,s)}(x_0,z) \pd{z}\\
&\geq c  (h(x_0)-h(x)) \int_{B(x,r)}g(|z-x_0|+s)\pd{z}\mathbb{E}^{x_0}\tau_{B(x_0,s)}\\
&\geq c  (h(x_0)-h(x)) |B(x,r)|g(|x-x_0|+2r)\mathbb{E}^{x_0}\tau_{B(x_0,s)}.
\end{align*}
This implies
$$-{\U}[u](x_0)\geq c  (h(x_0)-h(x)) |B(x,r)|g(|x-x_0|+2r)>0.$$
\end{proof}
For $r>0$, denote $s_r(x)= \E^{x}\tau_{B_r}$.
\begin{lemma}\label{lem:cont_s}
There exists a constant $c$ such that, for any $r>0$
\begin{align}\label{eq:ExtauEst}
 c^{-1}V(r-|x|)^2 \le s_r(x)  \le c V(r) V(r-|x|), \qquad |x|<r.
\end{align}
Furthermore, $s_r\in C_0(B_r)$.
\end{lemma}

\begin{proof}

By  the  slight modification of the proof of \cite[Lemma $2.9$]{MR3350043} we obtain $s_r\in C_0(B_r)$. 
For $x=0$ \eqref{eq:ExtauEst} follows by \eqref{eq:ExtauDest}. So, let $x\not=0$. Again, by \eqref{eq:ExtauDest}, we have
\begin{align*}
s_r(x) \ge \E^{x}(\tau_{B(x,|x|-r)}) \ge c V^2(|x|-r).
\end{align*}
To get the upper bound let $\theta = x/|x|$ and $H_r = \{ y \in \R^d \colon -r< y \cdot \theta <r\}$. Then, $s_r(x) \le \E^{x}\tau_{H_r}(x) = \E^{|x|}(\tau^Z_{(-r,r)})$, where $\tau^Z$ is the first exit time of a L\'evy process $Z = \theta \cdot X$ with characteristic exponent $\psi_Z(y) = \psi(\theta y)$, $y\in\R$.  Now, by \cite[Proposition 3.5]{MR3007664} we get $ \E^{x}\tau_{H_r}(x) \le cV_Z(r) V_Z(r-|x|)$, where $V_Z$ is a function constructed for the process $Z$ in the same way as $V$ is constructed for $X$. Now, by \cite[Proposition 2.4]{MR3350043} $V_Z \approx V$ and we get the desired result.
\end{proof}

\begin{lemma}\label{lem:cont_G}
For $0<r<R$ the function $y\mapsto G_{B_r}(0,y)$ is continuous on $\R^d\setminus\{0\}$.
\end{lemma}
\begin{proof}
First, let us observe that by \eqref{eq:Hunt_ap} it remains to prove that $P_{B_r}[G_{B_R}(0,\cdot)]$ is continuous. We note that since $G_{B_R}(0,\cdot)$ is continuous, it is bounded on $\R^d\setminus B_r$. Therefore, the remaining part can be proved by repeating the proof of \cite[Theorem 1.3]{MR1329992}, where instead of Proposition 1.19 therein one should use \cite[Proposition 4.4.1]{MR2152573}.
\end{proof}
By Lemma \ref{lem:cont_G}, $G_{B_r}(0,\cdot)$ is continuous on $\partial B_r$. We will use it to obtain the following decay of the Green function near the boundary. 
\begin{lem}\label{lem:GreenExit}
Let  $0<r<R$. Then,  
$$G_{B_r}(0,y)\leq CV(r-|y|), \quad |y|>3r/4,$$
where $C=c(d)V(r)\left(r^{-d}+V^{-2}(r)\sup_{|z|\geq r/2}G_{B_r}(0,z)\right)$.
\end{lem}
\begin{proof}
Fix $r>0$  and $w(y)=G_{B_r}(0,y)$. Observe that by the strong Markov property ${\U}[s_r](x)=-1$, $x\in B_r$. Denote $w_n(y)=w(y)\wedge n$. We take $n$ so large that $w_n(y)=w(y)$ on $B_r\setminus B_{r/2}$ (it is possible by the continuity of $G_{B_R}(0,\cdot)$).  Since $w$ satisfies MVP in $B_r\setminus \overline{B}_{3r/4}$, by \eqref{eq:G1} and \eqref{eq:ExtauEst} we have
\begin{align*}
{\U}[w_n](v)&={\U}[w_n-w](v)={\cal L}[w_n-w](v)=\int_{B_{r/2}}(w_n(z)-w(z))\nu(z-v) \pd{z}&\\
&\geq -cg(r/4)\int_{B}w(z) \pd{z} \geq -c_1 g(r/4)V^2(r),&  3r/4 \leq |v|<r.
\end{align*}
That is for $a>c_1 g(r/4)V^2(r)$,
$${\U}[as_{r} - w_n](v)<0,\qquad 3r/4 \leq |v|<r.$$
Furthermore, by \eqref{eq:ExtauEst}
$$s_r(v)\geq c V^2(r-|v|)\geq c_2 V^2(r), \qquad |v|<3r/4.$$
Hence,
$as_r(v) - w_n(v)>0$ for $|v|<3r/4$ if $a>n/(c_2V^2(r))$. 

Since $w_n = s_r=0$ on $B_r^c$, by the maximum principle for ${\U}$ (Lemma \ref{lem:maxpr}) and the continuity of $as_r-w_n$ on $\R^d \setminus \{0\}$, by Lemmas \ref{lem:cont_s} and \ref{lem:cont_G}, we obtain that $w_n\leq a s_r$ for $a=c_1 g(r/4)V^2(r)+n/(c_2V^2(r))$. Since $w_n(y)=w(y)$ on $B_r\setminus B_{r/2}$ and $g(r)\leq c(d)r^{-d}V^{-2}(r)$ we get 
$$G_{B_r}(0,y)\leq Cs_r(y), \quad |y|>3r/4,$$
where $C=c(d)\left(r^{-d}+V^{-2}(r)\sup_{|z|\geq r/2}G_{B_r}(0,z)\right)$. Now, the assertion follows by \eqref{eq:ExtauEst}.

\end{proof}

We define 
$$\bar{P}_r(z)=4r^{-1}\int^{r/2}_{r/4}P_{B_s}(0,-z)ds, \quad z\in \Rd.$$
The function $\bar{P}_r$ is called a regularized Poisson kernel. In the next lemma we show that like $P_{B(0,r)}(0,z)$, it reproduces functions satisfying MVP but in contrary to the latter, $\bar{P}_r$ is bounded.

\begin{lem}\label{lem:PoissonMean} Let $0<r\leq1$. Then,  $\bar{P}_r$ is bounded and  for every $u$ satisfying MVP in $B(x_0,r)$ we have
$$u(x)=\bar{P}_r*u(x),\quad |x-x_0|<r/2.$$
Furthermore, if $g^*$ is positive 
\begin{align}\label{eq:Prest}
\bar{P}_r(z)\leq c g^*(|z|)\mathbf{1}_{|z|\ge r/4}, \qquad z\in \Rd.
\end{align}
If $g$ is only nonnegative, \eqref{eq:Prest} holds only for large $|z|$.
\end{lem}
\begin{proof}
Fix $x_0 \in \R^d$ and let $u_0(x) = u(x+x_0)$. Then, $u_0$ satisfies MVP on $B_r$. Note that $P_{B_s}(0,-z)=0$ for $|z|\leq s$ and $P_{B_s}[|u|]<\infty$, $s<r$. For $|x-x_0|<r$, we have
\begin{align*}
\frac{r}{4}\bar{P}_r*u(x)&=\frac{r}{4}\bar{P}_r*u_0(x-x_0) = \int_{\Rd} \int^{r/2}_{r/4}P_{B_s}(0,z-x+x_0)u_0(z)ds dz \\
&= \int^{r/2}_{r/4}\int_{\Rd}P_{B(x-x_0,s)}(x-x_0,z)u_0(z)dzds
=   \int^{r/2}_{r/4}u_0(x-x_0)ds= \frac{r}{4} u(x).
\end{align*}
It remains to prove bounds for $\bar{P}_r$. For $|z|<r/4$ we have $\bar{P}_r(z)=0$. Let  $r/4<|z|<r$. Observe that for $r/4<s<r/2$, by Lemmas  \ref{lem:GreenExit}, \ref{lem:cont_s} and \ref{lem:cont_G} we have
\begin{align*}P_{B_s}(0,z)&\leq cg(r/16)\int_{\{|y|\leq 3s/4 \}\cup\{ |z-y|\geq r/16 \}}G_{B_s}(0,y)dy + C\int_{|y|>3s/4,|z-y|<r/16}V(s-|y|)g(z-y)dy\\
&\leq c_1g(r/16)V^2(r)+c_2\int_{|y|>3s/4,|z-y|>r/16}V(s-|y|)g(z-y)dy,\end{align*}
where $c_2=c(d)V(r)\left(r^{-d}+V^{-2}(r)\sup_{|z|\geq r/8}G_{B_r}(0,z)\right)<\infty$.
Hence,
\begin{align*}
\frac{r}{4}\bar{P}_r(z)&\leq c_1rg(r/16)V^2(r)/4+c_2\int^{r/2\wedge|z|}_{r/4}\int_{3s/4< |y|< s}1_{B(z,r/16)}(y)V(s-|y|)g(z-y)dyds\\
&= crg(r/16)V^2(r)+c_2\int_{3r/16< |y|< r/2\wedge |z|}\int^{4|y|/3\wedge r/2\wedge|z|}_{|y|}V(s-|y|)ds1_{B(z,r/16)}(y)g(z-y)dy\\
&\leq crg(r/16)V^2(r)+c_2\int_{3r/16< |y|<  |z|}\int^{|z|}_{|y|}V(s-|y|)ds1_{B(z,r/16)}(y)g(z-y)dy\\
&\leq cg(r/16)rV^2(r) +c_2\int_{ |y-z|<r/16}(|z-y|)V( |z-y|)g(z-y)dy
\end{align*}
By \cite[Proposition 3.4]{MR3350043} we obtain
$$\bar{P}_r(z)\leq c g(r/16)V^2(r)+ c_3
\frac{1}{V(r)}<\infty,\quad |z|<r.$$
For $|z|\geq r$, by \eqref{eq:PoissonUp} we have $P_{B_s}(0,z)\leq cV^2(s)g(|z|-s)$. If $g^*$ is positive, 
$g(|z|-s)\leq cg^*(|z|)$, which gives \eqref{eq:Prest} in this case. If $g^*(u)=0$, then  $\bar{P}_r(z)=0$, for $|z|\geq u+r/2$. 
\end{proof}

\begin{lemma}\label{lem:contMVP}
Let $u \in \Y$. If $u$ satisfies MVP in an open set $D$, then $u \in C(D)$. 
\end{lemma}
\begin{proof}
Let $\eps >0$. Since $u \in \Y$, there is $R>0$ such that $\|u\mathbf{1}_{B_R^c}\|_\Y<\eps$. Fix $x_0 \in D$ and let $\rho = 1 \land \dist{x_0,D^c}$ and $B = B(x_0,\rho)$.
Let $u_1 = u\mathbf{1}_{B_R}$ and $u_2 = u-u_1$. By Lemma \ref{lem:PoissonMean}, we have
\begin{align*}
u(x) = \bar{P}_\rho \ast u(x) = \bar{P}_\rho \ast u_1(x) + \bar{P}_\rho \ast u_2(x), \qquad |x-x_0|< \rho/2.
\end{align*}
Since $\bar{P}_\rho$ is bounded and $u_1 \in L^1(\R^d)$, $\bar{P}_\rho \ast u_1 \in C(\R^d)$. 

If there is $r_0$ such that $g^*(r) = 0$ for $r>r_0$, then $\bar{P}_\rho(z) = 0$ for $|z|>\rho+r_0$. Thus, by taking sufficiently large $R$, we get $\bar{P} \ast u_2(x) =0$ on $B$. Now, suppose $g^* >0$ on $(0,\infty)$.  By Lemma \ref{lem:PoissonMean}, $\bar{P}_\rho(z) \le c(1 \land g^*(|z|))$. 
Furthermore, there is a constant $c_1 = c_1(B)$ such that $1 \land g^*(|y-x|) \le c_1(1\land g^*(|y|)$ for $x \in B$. Then,
\begin{align*}
|\bar{P}_\rho \ast u_2(x)| \le \int_{B_R^c} \bar{P}_\rho(x-y)|u(y)| \pd{y} \le c_1c_2 \int_{B_R^c} |u(y)| (1 \land g^{*}(y) ) \pd{y} \le c_1c_2 \eps, \qquad x\in B. 
\end{align*}
Hence, for $x \in B$ we have 
\begin{align*}
|u(x_0) - u(x)| \le 2 c_1c_2 \eps + |\bar{P}_\rho \ast u_1(x_0) - \bar{P}_\rho \ast u_1(x)|,
\end{align*}
which gives the continuity of $u$ at $x_0$.
\end{proof}

\begin{lem}\label{lem:harm1}
Assume $u\in \Y\cap C^2(D)$.  If $u$ is $\L$-harmonic in  $D$, it satisfies MPV inside $D$.
\end{lem}
\begin{proof} 
Fix $u \in \Y \cap C^2(D)$, so ${\cal L} u(x)$ can be calculated pointwise for $x \in D$ and consequently, ${\cal L} u (x) = 0$ for $x \in D$. Fix an open set $D_1$ relatively bounded in $D$ and define $\widetilde{u}(x)=P_{D_1}[u](x)$, $x\in \Rd$.  We need too prove that $\widetilde{u} =u$. Since every open set may be approximated by bounded smooth domains (see e.g. \cite[Proposition 8.2.1]{MR2569323}, by the strong Markov property we may and do assume that $D_1$ is a $C^{1,1}$ domain. 

We claim that $\widetilde{u}$ has the mean-value property in $D_1$. To prove it, we only need to show $P_{D_1}[|u|](x)<\infty$. Let $D_2$ be an open set relatively compact in $D$ such that 	$\overline{D_1} \subset D_2$. There continuous  exist functions $u_1$, $u_2$ 
on $D_1^c$ such that $u=u_1+u_2$, $u_1$ is bounded on $D_1^c$ 
and $u_2 \equiv 0$ in $D_2$. We have
		\begin{align*}
	\widetilde{u}(x) = P_{D_1}[u_1](x)+P_{D_1}[u_2](x), 
	\quad x \in \Rd.
	\end{align*}
Note that $u_1 = P_{D_1}[u_1]$ and $u_2 = P_{D_1}[u_2]$ on $D_1^c$. Since $P_{D_1}[u_1]$  satisfy MVP in $D_1$, by Lemma \ref{lem:contMVP} it is continuous in $D_1$. Note that $P_{D_1}[u_1]$ is also continuous as a function of $x$ in 
$\overline{D_1}$. Indeed, let $x_0 \in \partial D$. For $\epsilon>0$ there exists 
$\delta>0$ such that
\begin{align*}
\left| \int_{D_1^c} P_{D_1}(x,z)u_1(z) dz - u_1(x_0) \right| \leq \epsilon + 
2 ||u_1||_{\infty} \mathbb{P}^x \(\left| X_{\tau_{D_1}}-x_0 \right| >\delta \) .
\end{align*}   

Since the second term goes to $0$ as $x \to x_0$ (see \cite[Lemmas 2.1 and 
2.9]{MR3350043}), by arbitrary choice of $\epsilon$ we get the continuity on $\overline{D_1}$.

Furthermore, from monotonicity of $1 \wedge g^*(|x|)$, the Ikeda-Watanabe formula and \eqref{eq:G1} we obtain
	\begin{align} \label{eq:Poissest}
	P_{D_1}(x,z) \leq \big(1 \wedge g^*(\dist{z,D_1} \big) \E^x 
\tau_{D_1}, 
	\quad x \in D_1, \ z \in D_2^c.
	\end{align}
	
Since $u \in \Y$,  and the fact that $g^*(r+1)\approx g^*(r)$, we get $P_{D_1}[|u_2|]<\infty$. Hence $P_{D_1}[u_2]$ satisfies MVP and by Lemma \ref{lem:contMVP}, $P_{D_1}[u_2]$ is continuous in $D_1$. Since $\E^x \tau_{D_1} \in C_0(D_1)$ and $u_2 = 0$ on $\partial D_1 \subset D_2$, by \eqref{eq:Poissest} $P_{D_1}[u_2]$ is continuous in $\overline{D_1}$ as well. Hence, $\widetilde{u}$ is continuous and has the mean-value property in $D_1$. Note that  $\widetilde{u}=u$ on $D_1^c$, since $D_1$ is a $C^{1,1}$ domain. 

\medskip
	
Let $h=\widetilde{u}-u$. We now verify that $h \equiv 0$ so 
that $u=\widetilde{u}$ has the mean-value property in $D_1$. Observe that $h$ is continuous and compactly supported.  Suppose $x_0 \in D_1$ is such that $h(x_0)=\sup_{x \in D_1}h(x)>0$. Since $u\in C^2(D)$ we have $0={\cal L} u(x_0)={\U}[u](x_0)$. Hence, by Lemma \ref{lem:maxpr}, $h$ is constant on $\mathrm{supp}(\nu)+x_0$.  If $D_1\subset \mathrm{supp}(\nu)+x_0$ we get that $h>0$ on $D_1$, which is a contradiction due to continuity of $h$ and the fact that $h \equiv 0$ on $\partial D_1$. Thus, $h \leq 0$. If this is not the case, then we can use the chain rule to get for any $n\in\N$ that $h$ is constant on $n\mathrm{supp}(\nu)+x_0$ and eventually get a contradiction. Similarly, $h$ must be 
non-negative.
\end{proof}

\begin{lem}\label{lem:harm=MVP}
	Let $u \in Y$ be a solution of ${\cal L} u=0$ in $D$ in 
distributional sense. Then $u$ satisfies MPV inside $D$.
\end{lem}

\begin{proof} 
	Let $\Omega \subset \overline\Omega\subset D$ be a bounded $C^{1,1}$ domain.  
Define $\rho = (1 \wedge \dist{\Omega,D^c})/2$ and let $V=\Omega+B_{\rho}$. For $\epsilon < 
\rho/4$ we consider standard mollifiers $\phi_{\epsilon}$. 
Since ${\cal L}$ is translation-invariant we have that ${\cal L}(\phi_{\epsilon} \ast 
u)= {\cal L} u \ast \phi_{\epsilon}=0$ in $V_{\epsilon}=\{ x \in D\colon 
\dist{x,V^c}>\epsilon \}$ 
	in distributional sense. By Lemma
\ref{lem:harm1} we 
	obtain
	\begin{align*}
	\phi_{\epsilon} \ast u(x) = P_\Omega [ \phi_{\epsilon} \ast u] (x), \quad x 
	\in \Omega.
	\end{align*}
	Note that by \cite[Lemma 2.9]{MR4194536}, for $u \in {\Y}$, $\phi_{\epsilon} \ast u \to u$  in 
$\Y$ as $\epsilon \to 0^+$.  
	Hence, up to the subsequence,
	\begin{align*}
	\lim\limits_{\epsilon \to 0^+} \phi_{\epsilon} \ast u(x) = u(x) \quad 
	\text{a.e.}
	\end{align*}
	Furthermore, since $\phi_{\epsilon} \ast u$ satisfies MVP in $V_{\rho/4}$, 
	 by Lemma \ref{lem:PoissonMean}
	$$\phi_{\epsilon} \ast u(x)= \phi_{\epsilon} \ast u\ast \bar{P}_r(x), \quad x \in V_{\rho/2}$$
	for $r=\rho/8$.  Hence, for any $E \subset \Omega^c$, 
	\begin{align*}
	P_{\Omega}[\left| \phi_{\epsilon} \ast u \right| \mathbf{1}_{V_{\rho/2}\cap E}](x) &= 
\int_{V_{\rho/2} \cap \Omega^c\cap E} \left| \phi_{\epsilon} \ast u(z) \right| 
P_\Omega(x,z) \pd{z} \\ &= \int_{V_{\rho/2}\cap \Omega^c \cap E} \left| 
\phi_{\epsilon} \ast u \ast \bar{P}_{r}(z) \right| P_\Omega(x,z) \pd{z} 
\\ 
&\leq \int_{B_{\epsilon}} \phi_{\epsilon}(v) \int_{\Rd} \left| u(y) \right| 
\int_{V_{\rho/2} 
\cap \Omega^c \cap E} \bar{P}_{r}(z-y-s)P_\Omega(x,z) \pd{z} \pd{y} \pd{v}. 
	\end{align*}
	Let $R>0$.	Then from boundedness of $\bar{P}_{r}$ and local integrability of $u$ we get
	\begin{align*}
	\int_{|y|\leq R} |u(y)| \int_{V_{\rho/2} \cap \Omega^c \cap E} 
	\bar{P}_{r}(z-y-s)P_\Omega(x,z) \pd{z} \pd{y} &\leq c \int_{|y|\leq c} 
|u(y)| dy \int_{E} P_\Omega(x,z) \pd{z} \\ &\leq c ||u||_{\Y} \int_{E} 
P_\Omega(x,z) \pd{z}.
	\end{align*}

	Furthermore, for $|y|>R$ and sufficiently large $R$, we have $\bar{P}_{r}(z-y-v) \leq c (1\wedge g^*(|y|))$. 
	Thus,
	\begin{align*}
	\int_{|y|>R} |u(y)| \int_{V_{\rho/2}\cap \Omega^c \cap E} 
	\bar{P}_{r}(z-y-s)P_\Omega(x,z) \pd{z}\pd{y} 
	&\leq c||{u}||_{\Y} \int_E P_\Omega(x,z) \pd{z}.
	\end{align*}
	
Since $\int_E P_\Omega(x,z) \pd{z} \to 0$ as measure of $E$ tends to 0, $\phi_{\epsilon} \ast u$ are uniformly integrable in $V_{\rho/2}$ 
with respect to the measure $P_\Omega(x,z)\pd z$. By the Vitali 
convergence theorem,
	
\begin{align*}
\lim_{\epsilon \to 0^+} P_\Omega [ (\phi_{\epsilon} \ast u) \mathbf{1}_{ V_{\rho/2}} ](x) 
= P_\Omega [u \mathbf{1}_{V_{\rho/2}} ](x).
\end{align*}

It remains to show that $\lim_{\epsilon \to 0^+} P_\Omega [ (\phi_{\epsilon} \ast u) \mathbf{1}_{ V_{\rho/2}^c} ]
= P_\Omega [u \mathbf{1}_{V_{\rho/2}^c} ]$. 
Since $\dist{\Omega,V_{\rho/2}^c)}=\rho/2$, we obtain, by \eqref{eq:PoissonUp},  $P_\Omega(x,z) \le c (1 \land g^*(z))$. Due to  $\lim_{\epsilon \to 0^+} 
\phi_{\epsilon} \ast u = u$ in $\Y$ we obtain the result.
\end{proof}

\bibliographystyle{abbrv}
\bibliography{sources/dhkgpuL}

\noindent Tomasz Grzywny\\
Faculty of Pure and Applied Mathematics, Wroc\l{}aw University of Science and Technology, Wyb. Wyspiańskiego 27, 50-370 Wrocław,
Poland\\
\emph{email address}: tomasz.grzywny@pwr.edu.pl\\\\

\noindent Tomasz Jakubowski\\
Faculty of Pure and Applied Mathematics, Wroc\l{}aw University of Science and Technology, Wyb. Wyspiańskiego 27, 50-370 Wrocław,
Poland\\
\emph{email address}: tomasz.jakubowski@pwr.edu.pl\\\\

\noindent Grzegorz \.Zurek\\
Faculty of Pure and Applied Mathematics, Wroc\l{}aw University of Science and Technology, Wyb. Wyspiańskiego 27, 50-370 Wrocław,
Poland\\
\emph{email address}: grzegorz.zurek@pwr.edu.pl

\end{document}